\documentclass[12pt,a4paper]{amsart}
\usepackage{amsfonts}
\usepackage{amsthm}
\usepackage{amsmath}
\usepackage{amscd}
\usepackage[utf8]{inputenc}
\usepackage{bbold}
\usepackage{t1enc}
\usepackage[mathscr]{eucal}
\usepackage{indentfirst}
\usepackage{graphicx}
\usepackage{graphics}
\usepackage{pict2e}

\numberwithin{equation}{section}

\theoremstyle{plain}
\newtheorem{Th}{Theorem}[section]
\newtheorem{Lemma}[Th]{Lemma}

\newtheorem{Pro}[Th]{Proposition}

 \theoremstyle{definition}
\newtheorem{Def}[Th]{Definition}

\newtheorem{Rem}[Th]{Remark}
\newtheorem{?}[Th]{Problem}

\DeclareMathOperator{\sgn}{sgn}

\newcommand{\C}{\mathbb{C}}

\newcommand{\Gr}{\mathbb{Gr}}

\begin{document}
\title{Matrix tuples with linearly dependent invariant subspaces}
\author{Tam\'as Bencze}
\address{E\"{o}tv\"{o}s Lor\'{a}nd University,
  P\'{a}zm\'{a}ny P\'{e}ter s\'{e}t\'{a}ny 1/C, Budapest, 1117 Hungary}
\email{ benczetamas11@gmail.com}
\begin{abstract}
The set of matrix tuples with invariant subspaces whose dimensions sum up to the dimension of the space, but which do not span the whole space form an algebraic hypersurface. We found the equation of this hypersurface. This generalizes previous joint work.
\end{abstract}
\maketitle
\section{Introduction}
In our previous joint work with Péter E.\ Frenkel \cite{tri},  we found the equation of the set of matrix pairs that have nontrivially intersecting invariant subspaces of complementary dimensions. In this article I generalize this result to several matrices with invariant subspaces whose dimensions sum up to the dimension of the space, but which do not span the whole space. For  easier understanding, I will show the proof for the special case of 1 dimensional invariant subspaces before proving the general case.

\subsection*{Notations}
We will call $\mathbb{C}^n$ the set of column vectors, and $e_1$, $\dots$, $e_n$ its standard basis. Unless otherwise stated, vectors are column vectors. We will call $(\C^n)^*$ the set of row vectors, and $e_1^*$, $\dots$, $e_n^*$ its standard basis. We will call the ring of $n\times n$ complex matrices $M_n(\C)$. We write $\Gr_k(n)$ for the Grassmannian variety whose points parametrize the $k$-dimensional linear subspaces of $\C^n$. We will denote the polynomial ring with variables $x_1$, $\dots$, $x_n$ by $\C[x_1, \dots, x_n]$, and its subset of degree $d$ homogeneous polynomials by $\C[x_1, \dots, x_n]_d$. We will denote the set $\{1, \dots, n\}$ by $[n]$.
\section{The case of $n$ matrices in $M_n(\mathbb{C})$}
\begin{Def}
For $A_i\in M_n(\mathbb{C})$ $(i\in [n])$, we define $A$ on the tensor product space $(\mathbb{C}^n)^{\otimes n}$ by the following equation:\\
$A(v_1\otimes\dots\otimes v_n)=\sum\limits_{i=1}^n v_1\otimes \dots\otimes A_i(v_i)\otimes\dots\otimes v_n$
\end{Def}
\begin{Rem}
If $A_i$ has eigenvalues $\lambda_{i,j}$ ($i,j\in [n]$), then $A$ has eigenvalues $\sum\limits_{i=1}^n\lambda_{i,f(i)}$ ($f:[n]\rightarrow[n]$)
\end{Rem}
\begin{Def}\label{w}
$$M(A_1,\dots,A_n):=\begin{pmatrix} w\\ wA\\ \vdots \\wA^{n^n-1}\end{pmatrix}$$ where $$w:=\sum\limits_{\pi\in S_n}\sgn(\pi)\bigotimes\limits_i e_{\pi(i)}^*$$ 

$$P(A_1,\dots,A_n):= \det M(A_1,\dots,A_n)$$
\end{Def}
\begin{Lemma}\label{pol}
For $v\in (\mathbb{C}^n)^*, B\in M_n(\mathbb{C})$, $$\det \begin{pmatrix} v\\ vB\\ \vdots \\vB^{n-1}\end{pmatrix}=0$$ if and only if $B$ has an eigenvector orthogonal to $v$.
\end{Lemma}
\begin{proof}
If the determinant is 0, the matrix has a nonzero vector $u$ in its kernel. This means that $vB^ku=0$ for all $k<n$, hence the whole $B$-invariant subspace generated by $u$ is orthogonal to $v$. It contains an eigenvector.

For the other direction, if $B$ has an eigenvector $u$ that is orthogonal to $v$, then $vB^ku=0$ for every $k$, so $u$ is in the kernel of the matrix, proving that the determinant is $0$.
\end{proof}
\begin{Pro}
If there exist an eigenvector $v_i$ of $A_i$ for each $i$, such that $v_1, \dots, v_n$ are linearly dependent, then $P(A_1,\dots,A_n)=0$
\end{Pro}
\begin{proof}
$v:=\bigotimes v_i$
\\
It's easy to check that $$wv=\det(v_1,v_2,\dots ,v_n)=0$$ and $v$ is an eigenvector of $A$, so we can use \ref{pol}.
\end{proof}
The converse isn't always true, but if $A$ has no multiple eigenvalue, it is:
\begin{Pro}
If $A$ has no multiple eigenvalue and $P(A_1,\dots,A_n)=0$, then there exist an eigenvector $v_i$ of $A_i$ for each $i$, such that $v_1, \dots, v_n$ are linearly dependent.
\end{Pro}
For the proof, we first need a lemma:
\begin{Lemma}
If $A$ has no multiple eigenvalue, then all of its eigenvectors are of the form $\bigotimes v_i$ for some eigenvector $v_i$ of $A_i$ (for all $i\in [n]$)
\end{Lemma}
\begin{proof}
If $A$ has no multiple eigenvalue, then none of the $A_i$ has either, so $A_i$ has an eigenbasis $\{v_{i,j}|j\in[n]\}$. The tensor product $\bigotimes\limits v_{i,f(i)}$ is an eigenvector of $A$ for all $f:[n]\rightarrow [n]$ (and they are linearly independent), so they form an eigenbasis of $A$. As $A$ has no multiple eigenvalue, each eigenvector must be a scalar multiple of one of these vectors.
\end{proof}
\begin{proof}
By \ref{pol}, $A$ has an eigenvector orthogonal to $w$, and by the previous lemma, that eigenvector is of the form $\bigotimes v_i$. Then $$0=w(\bigotimes v_i)=\det(v_1,v_2,...v_n)$$ so the $v_i$ are linearly dependent.
\end{proof}
Let's look at the polynomial $$\prod\limits_{f\neq g:[n]\rightarrow[n]}\sum\limits_{i\in [n]}(\lambda_{i,f(i)}-\lambda_{i,g(i)})$$

It decomposes into $$\prod\limits_{\emptyset\neq K\subseteq[n]}\left(\prod\limits_{\substack{f,g:K\rightarrow [n]\\ \forall i f(i)\neq g(i)}}\sum\limits_{i\in K}(\lambda_{i,f(i)}-\lambda_{i,g(i)})\right)^{n^{n-|K|}}$$

The polynomial in the brackets is invariant under permutations of the $\lambda_{i,j}$ with given $i$, so there exists a unique polynomial in the entries of the $A_i$ that is equal to this polynomial of the eigenvalues, and it is invariant under conjugation (of any $A_i$). Let us call this polynomial $d_K$. 

\begin{Lemma}
$d_K$ is irreducible if $|K|=1$, and the square of an irreducible if $|K|>1$
\end{Lemma}
\begin{proof}
Suppose that $d_K$ has a divisor. Since $d_K$ is conjugation invariant, if we conjugate an input matrix with an invertible matrix before applying this divisor, we still get a divisor of $d_K$. This action gives a group homomorphism from $GL_n(\mathbb{C})^{|K|}$ to the permutation group on these divisors, and this map has a transitive image. As $d_K$ has only finitely many divisors, this is a finite group, and since $GL_n(\mathbb{C})^{|K|}$ has no nontrivial finite factor groups, our divisor is fixed by all conjugations, so it can be expressed as a polynomial of the eigenvalues $\lambda_{i,j}$, which is invariant under permutations for fixed $i$. 

In the polynomial $$\prod\limits_{\substack{f,g:K\rightarrow [n]\\ \forall i f(i)\neq g(i)}}\sum\limits_{i\in K}(\lambda_{i,f(i)}-\lambda_{i,g(i)})$$ each factor appears twice (as switching $f$ and $g$ only changes the sign), and on the different factors, the permutations act transitively, so $d_K$ can only have its square root as a nontrivial divisor. Take the square root: the product of those factors where $f$ is lexicografically smaller than $g$. As any permutation can only change the sign (its square is invariant), it is enough to check if switching $\lambda_{i,1}$ and $\lambda_{i,2}$ changes the sign for any $i$. Switching them changes the sign of $(n(n-1))^{|K|-1}$ factors if $i$ is the smallest element of $K$ (the factors with $f(i)=1$ and $g(i)=2$), and $0$ factors otherwise. $(n(n-1))^{|K|-1}$ is even if $|K|>1$, so switching $\lambda_{i,1}$ and $\lambda_{i,2}$ does not change the sign of the polynomial, but it is odd if $|K|=1$, so in this case we do not get a permutation invariant polynomial.
\end{proof}
 Let's call this irreducible polynomial $D_K$.

\begin{Rem}
For different $K$, $D_K$ is a different irreducible, since $D_K$ depends on the matrix $A_i$ if and only if $i\in K$.
\end{Rem}

\begin{Pro}
$P$ is divisible by $D_K^{n^{n-|K|}}$, if $|K|>1$.
\end{Pro}
\begin{proof}
By \cite[Lemma 3.2.]{tri}, it is enough to prove that \\$D_K=0\Rightarrow \dim \ker M\geq n^{n-|K|}$\\
And because $\dim \ker M\geq n^{n-|K|}$ is a Zariski-closed set, we may assume that $D_{K'}\neq 0$ for all $|K'|=1$, so $A$ is diagonalizable.

Since $D_K=0$, $A$ has $n^{n-|K|}$ distinct pairs of equal eigenvalues. The 2 dimensional subspaces spanned by the corresponding eigenvectors each contain an eigenvector orthogonal to $w$, hence in the kernel of $M$.
\end{proof}
\begin{Lemma}
The set $$Y=\{(v_1,v_2,\dots,v_n)\in \mathbb{P}(\mathbb{C}^n)^n| v_1,\dots ,v_n \textrm{ linearly dependent}\}$$ is an irreducible projective variety.
\end{Lemma}
\begin{proof}
This set is the vanishing set of the irreducible multihomogeneous polynomial $\det(v_1,v_2,\dots,v_n)$.
\end{proof}
\begin{Lemma}\label{eigenvector}
The set $$\{(A,v)\in \mathbb{P}M_n(\mathbb{C})\times\mathbb{P}(\mathbb{C}^n)|v\textrm{ is an eigenvector of }A\}$$ is a projective variety.
\end{Lemma}
\begin{proof}
The $2\times 2$ minors of the $n\times 2$ matrix $\begin{pmatrix}v & Av\end{pmatrix}$ are bihomogeneous polynomials in the entries of the matrix $A$ and the vector $v$. Their simultaneous vanishing defines the point set in question.
\end{proof}
\begin{Pro}\label{n-irred}
The set $X$ of $n$-tuples $(A_1,A_2,\dots ,A_n)\in (\mathbb{P}M_n(\mathbb{C}))^n$, such that there exists $v_1,\dots, v_n$ linearly dependent vectors, with $v_i$ eigenvector of $A_i$, form an irreducible projective variety.
\end{Pro}
\begin{proof}
Let's consider the set of $2n$-tuples $(A_1,A_2,\dots ,A_n,v_1,v_2,\dots, v_n)$, with $v_i$ being an eigenvector of $A_i$, and $v_1,\dots,v_n$ dependent. By the previous lemmas, this set is a projective variety, and since it is a projective space bundle over $Y$, it is irreducible. The set in the proposition is the projection of this set to $(\mathbb{P}M_n(\mathbb{C}))^n$, proving the proposition.
\end{proof}
\begin{Lemma}\label{squarefree}
(a) For any bihomogeneous polynomial $Q$ in $a+b$ complex variables, the set $$\{v\in \mathbb{C}^a|Q(v,x) \textrm{ is not square-free as a polynomial in } x\}$$ is an affine variety. 

(b) If $Q$ is square-free, then this set is not the whole $\mathbb{C}^a$.
\end{Lemma}
\begin{proof}
For fixed $v$, let $Q_v(x)=Q(v,x)$, $x=(x_1,\dots,x_b)$. 

(a) $Q_v$ is not square-free if and only if $$\gcd\left(Q_v, \frac{\partial}{\partial x_1}Q_v, \dots, \frac{\partial}{\partial x_b}Q_v\right)\neq 1.$$ (Proof: If $Q_v=\prod\limits_ip_i$ is the irreducible decomposition of $Q_v$, then $$\frac{\partial}{\partial x_j}Q_v=\sum \limits_{k}\frac{\partial}{\partial x_j}p_k\prod\limits_{i\neq k}p_i.$$ $p_l$ divides all terms except $\frac{\partial}{\partial x_j}p_l\prod\limits_{i\neq l}p_i$. If $p_l$ appears twice, then it divides this term too, so it divides the gcd. If it only appears once, then $p_l$ does not divide $\prod\limits_{i\neq l}p_i$, so it must divide $\frac{\partial}{\partial x_j}p_l$, which is a polynomial of lower degree, so it must be 0. So if $p_l$ divides the gcd, then all its partial derivatives are 0, a contradiction.)

Now we will prove that $$\left\{(p_1,\dots,p_l)\in\prod\limits_i\mathbb C[x_1, \dots, x_k]_{d_i}|\gcd(p_1,\dots,p_l)\neq 1\right\}$$ is a variety. (This would prove (a) as the coefficients of $Q_v$,  $\frac{\partial}{\partial x_1}Q_v$, \dots, $\frac{\partial}{\partial x_b}Q_v$ depend polynomially on $v$.)

This is because it is the union of the images of the multihomogeneous polynomial maps $$f_j:\mathbb C[x_1, \dots, x_k]_{j}\times\prod\limits_i\mathbb C[x_1, \dots, x_k]_{d_i-j}\rightarrow\prod\limits_i\mathbb C[x_1, \dots, x_k]_{d_i},$$ $f_j(p, p_1, \dots, p_l)=(pp_1, \dots, pp_l)$ as $j$ goes from $1$ to $\min\limits_i d_i$. We can think of $f_j$ as a map between products of projective spaces, and it is known that under a polynomial map between projective spaces, the image of a projective variety is a projective variety.

(b) Assume that the set is the whole $\mathbb{C}^a$. We will look at the set $$0=Q=\frac{\partial}{\partial x_1}Q= \dots= \frac{\partial}{\partial x_b}Q.$$ On the fiber over any point $v\in\mathbb{C}^a$, this set has codimension at most $1$ (since $Q_v, \frac{\partial}{\partial x_1}Q_v, \dots$ have a common divisor), so the whole set has codimension at most $1$, which means that $Q,\frac{\partial}{\partial x_1}Q, \dots, \frac{\partial}{\partial x_b}Q$ have a common divisor. $Q$ is square-free, so this means that $\frac{\partial}{\partial x_i}Q=0$ for all $i$, meaning that $Q$ does not depend on $x$. This is a contradiction, as constants are square-free.
\end{proof}
\begin{Pro}
The smallest degree polynomial defining the set $X$ has degree $n^n\binom{n}{2}$.
\end{Pro}
\begin{proof}
We will prove that this polynomial is homogeneous of degree $n^{n-1}\binom{n}{2}$ in each of the $A_i$ (this would prove the proposition). As this is symmetric, we will only prove it for $A_n$. Fix the other matrices generically: none of them have a multiple eigenvalue, there are no dependent vectors $v_1,\dots, v_{n-1}$ with $v_i$ being an eigenvector of $A_j$ for some $j\in [n-1]$, and fixing these matrices leaves the defining polynomial square-free as a polynomial of $A_n$ (we can do this by \ref{squarefree}). $(A_1,\dots,A_n)\in X\Leftrightarrow A_n$ has an eigenvector on one of the hyperplanes generated by $v_1,\dots, v_{n-1}$ with eigenvector $v_i$ of $A_i$. There are $n^{n-1}$ such hyperplanes, and for a given hyperplane $V$, the set of matrices $A_n$ having an eigenvector on it is an irreducible variety (it goes similarly to \ref{n-irred}): it is the projection of the set $$\{(A,v)\in \mathbb{P}M_n(\mathbb{C})\times\mathbb{P}(\mathbb{C}^n)|v \text{ is an eigenvector of } A\text{ on }V\},$$ which is a variety, because it is defined by the equations  in \ref{eigenvector} and the equation defining $V$, and is irreducible, since it is a projective space bundle over the projective subspace $V$. We will call this variety $S_V$.

Now we define $T(-,-)$ as in \cite{tri}, which is an irreducible homogeneous polynomial of degree $n\binom{n}{2}$ in the entries of the first variable matrix, and is $0$ if and only if the first matrix has an eigenvector that is contained in some $n-1$ dimensional invariant subspace of the second. If we fix a matrix $B$ with distinct eigenvalues and $T(-,B)$ square-free (we can do this by \ref{squarefree}), $T(-,B)=0$ if and only if the variable matrix is in $S_V$ for some $V$ $n-1$ dimensional invariant subspace of $B$. This means that $S_V$ is the vanishing set of some irreducible polynomial which we will call $Q_V$, and $T(-,B)$ is the product of these $Q_V$, where $V$ runs through all $n-1$ dimensional invariant subspaces of $B$. The degree of $Q_V$ is the degree of $T(-,B)$ divided by the number of these subspaces:$\binom{n}{2}$

The minimum degree polynomial defining $X$ (for fixed $A_1$, $\dots$, $A_{n-1}$) is the product of $Q_V$, where $V$ goes through all subspaces $\langle v_1$, $\dots$, $v_{n-1}\rangle$, where $v_i$ is an eigenvector of $A_i$.

This proves that the polynomial is homogeneous of degree $n^{n-1}\binom{n}{2}$ in $A_n$.
\end{proof}
We have already proved (\ref{pol}) that if the polynomial defining $V$ is $\sum v_ix_i=0$ for some $0\neq v=(v_1,\dots,v_n)\in (\C^n)^*$, then
$$B\in S_V\Leftrightarrow\det \begin{pmatrix} v\\ vB\\ \vdots \\vB^{n-1}\end{pmatrix}=0,$$ but now we can prove something stronger:
\begin{Rem} If the polynomial defining $V$ is $\sum v_ix_i=0$ for some $v=(v_1,\dots,v_n)\in (\C^n)^*$, then
$$Q_V(B)=\det \begin{pmatrix} v\\ vB\\ \vdots \\vB^{n-1}\end{pmatrix}$$
\end{Rem}
\begin{proof}
By \ref{pol}, we know that these two polynomials have the same vanishing set, they have the same degree, and $Q_V$ is irreducible.
\end{proof}
\begin{Th}
$X$ is the vanishing set of the irreducible polynomial $$\frac{P}{\prod\limits_{|K|>1}D_K^{n^{n-|K|}}}$$
\end{Th}
\begin{proof}
We already know that $X$ is the vanishing set of an irreducible divisor of $P$ with degree $n^n\binom{n}{2}$. As it clearly can't be $D_K$ for any $K$ (as shown by generic diagonal matrices), it is also a divisor of $\frac{P}{\prod\limits_{|K|>1}D_K^{n^{n-|K|}}}$. But this polynomial has degree $n^n\binom{n}{2}$, so it has to be the polynomial defining $X$.
\end{proof}
\section{The general case}
\begin{Def}
For $A_i\in M_n(\mathbb{C})$ $(i\in [l])$ and $k_i$ a partition of $n$, we define $A$ on the tensor product space $\bigotimes\limits_{i=1}^l(\mathbb{C}^n)^{\wedge k_i}$ by the following equation:\\
\begin{align*} A(\bigotimes\limits_{i=1}^l\bigwedge\limits_{j=1}^{k_i}v_{i,j})=\\\sum\limits_{i=1}^l\sum\limits_{j=1}^{k_i} \left((\bigotimes\limits_{p<i}\bigwedge\limits_{q=1}^{k_p}v_{p,q})\otimes((\bigwedge\limits_{q<j}v_{i,q})\wedge A_i(v_{i,j})\wedge(\bigwedge\limits_{q>j}v_{i,q}))\otimes(\bigotimes\limits_{p>i}\bigwedge\limits_{q=1}^{k_p}v_{p,q})\right)
\end{align*}
\end{Def}
\begin{Rem}
If $A_i$ has eigenvalues $\lambda_{i,j}$ $(i\in [l],j\in [n])$, then $A$ has eigenvalues $\sum\limits_{i=1}^l\sum \limits_{j\in F(i)}\lambda_{i,j}$ $(F:[l]\rightarrow 2^{[n]}, |F(i)|=k_i)$
\end{Rem}
\begin{Def}
$$M(A_1,\dots, A_l):=\begin{pmatrix} \bar w\\ \bar w A\\ \vdots \\ \bar w A^{\prod \binom{n}{k_i}-1}\end{pmatrix}$$ where $\bar w$ is the image of $w$ (from \ref{w}) under the canonical homomorphism $$(\mathbb{C}^n)^{\otimes n}\rightarrow \bigotimes\limits_{i=1}^l(\mathbb{C}^n)^{\wedge k_i}$$\\
$ P=\det M$
\end{Def}
\begin{Pro}
If there exists a $k_i$ dimensional invariant subspace $V_i$ of $A_i$ for each $i$, which together do not span the whole $\C^n$, then $P=0$. If $A$ has no multiple eigenvalue, the converse also holds.
\end{Pro}
\begin{proof}
Let $v_{i,j}$ ($j \in [k_i]$) be an eigenbasis of $V_i$, and $$v=\bigotimes\limits_{i=1}^l\bigwedge\limits_{j=1}^{k_i}v_{i,j}$$

We have $$\bar wv=\det(v_{1,1},\dots,v_{l,k_l})=0$$ and $v$ is an eigenvector of $A$, so we can use \ref{pol}.

For the converse, if $A$ has no multiple eigenvalue (so none of the $A_i$ have either), then each eigenvector of $A$ can be written in the form $\bigotimes\limits_{i=1}^l\bigwedge\limits_{j=1}^{k_i}v_{i,j}$, where $\langle v_{i,1},\dots v_{i,k_i}\rangle$ is an invariant subspace of $A_i$ for every $i$. By \ref{pol}, $A$ has an orthogonal eigenvector to $\bar w$, but $$\bar w\bigotimes\limits_{i=1}^l\bigwedge\limits_{j=1}^{k_i}v_{i,j}=\det(v_{1,1},\dots,v_{l,k_l})$$ hence it being $0$ means that the subspaces $\langle v_{i,1}, \dots, v_{i, k_i}\rangle$ do not span the whole space.  
\end{proof}
Let's look at the polynomial $$\prod\limits_{F\neq G}\sum\limits_{i\in [l]}\left(\sum\limits_{j \in F(i)}\lambda_{i,j}-\sum\limits_{j\in G(i)}\lambda_{i,j}\right)$$ where $F$ and $G$ run through all functions that map each $i\in [l]$ to a $k_i$ element subset of $[n]$.

It decomposes into $$\prod\limits_{\substack{0\leq k'_i\leq k_i\\\exists k'_i\neq 0}}\left(\prod\limits_{F,G}\sum\limits_{i\in[l]}\left(\sum \limits_{j\in F(i)}\lambda_{i,j}-\sum\limits_{j\in G(i)}\lambda_{i,j}\right)\right)^{\prod\binom{n-2k_i'}{k_i-k_i'}},$$ where $F$ and $G$ run through all functions that map each $i\in [l]$ to a $k'_i$ element subset of $[n]$, with $F(i)\cap G(i)=\emptyset$ for all $i$. The product in brackets is invariant with respect to permutations of the $\lambda_{i,j}$ with given $i$, so there exists a unique polynomial in the entries of the $A_i$, that is equal to this polynomial of the eigenvalues, and it is invariant under conjugation (of any $A_i$). Let us call this polynomial $d_{k'}$, where $k'=(k'_1, \dots, k'_l)$.

\begin{Lemma}
When $\sum k'_i=1$, $d_{k'}$ is irreducible, otherwise it is the square of an irreducible.
\end{Lemma}
\begin{proof}
Since $d_{k'}$ is conjugation invariant, conjugations permute its divisors, and this gives us a group action of $GL_n(\mathbb{C})^{l}$ on these divisors. The stabilizer of a divisor is a finite index subgroup of $GL_n(\mathbb{C})^{l}$, but this group has no nontrivial finite index subgroups, which means that all divisors are invariant under conjugation, so they can be written as  polynomials in the eigenvalues $\lambda_{i,j}$, which are invariant under permutations for fixed $i$.

In the product $$\prod\limits_{F,G}\sum\limits_{i\in[l]}\left(\sum \limits_{j\in F(i)}\lambda_{i,j}-\sum\limits_{j\in G(i)}\lambda_{i,j}\right),$$ each factor appears twice (as switching $F$ and $G$ only changes the sign), and on the different factors, the permutations of the eigenvalues act transitively, so $d_{k'}$ can only have its square root as its divisor. We take for its square root: the product of those factors where $F$ is lexicographically smaller than $G$ (for the smallest $i$ such that $k'_i\neq 0$, $\min(F(i))<\min(G(i))$). As any permutation can only change the sign, it is enough to check if switching $\lambda_{i,1}$ and $\lambda_{i,2}$ changes the sign for any $i$. It changes the sign of a given factor if and only if $\lambda_{i,1}\in F(i)$, $\lambda_{i,2}\in G(i)$, and $k'_j=0$ if $j<i$. This means that it changes the sign of $0$ factors if $i$ isn't the smallest number such that $k'_i>0$, and $\binom{n-2}{k'_i-1}\binom{n-k'_i-1}{k'_i-1}\prod\limits_{j\neq i}\binom{n}{k'_j}\binom{n-k'_j}{k'_j}$ factors otherwise.  If $\sum\limits_i k'_i=1$, this number is $1$, which means that switching $\lambda_{i,1}$ and $\lambda_{i,2}$ changes the sign of the polynomial, so this isn't a divisor; $d_{k'}$ is irreducible. If $\sum\limits_i k'_i>1$ , an even number of factors changed signs, which means that its square root really is a conjugation invariant divisor of $d_{k'}$.
\end{proof}
Let's call this irreducible polynomial $D_{k'}$.

\begin{Rem}
For different $k'$, $D_{k'}$ is a different irreducible.
\end{Rem}
\begin{proof}
The set $$\{(-2)^1, (-2)^2,\dots, (-2)^{2a-1}, -(2^{2a}-2), (-2)^{2a+1},\dots,(-2)^b\}$$ has the property that there is only one way to choose two disjoint nonempty subsets with the same sum, and these two sets both have $a$ elements. We can use these for $b=nl$, $a=\sum k'_i$ as the eigenvalues of the $A_i$ to construct matrices for which $D_{k'}=0$ but $D_{k''}\neq 0$ for any $k''\neq k'$.
\end{proof}

\begin{Pro}
$P$ is divisible by $D_{k'}^{\prod\binom{n-2k_i'}{k_i-k_i'}}$ if $\sum k'_i>1$.
\end{Pro}
\begin{proof}
By \cite[Lemma 3.2.]{tri}, it is enough to prove that $D_{k'}=0$ implies that $\dim \ker M\geq \prod\binom{n-2k_i'}{k_i-k_i'}$.\\
And because $\dim \ker M\geq \prod\binom{n-2k_i'}{k_i-k_i'}$ is a Zariski-closed set, we may assume that $D_{k''}\neq 0$ if $\sum k_i''=1$, therefore $A$ is diagonalizable.\\
Since $D_{k'}=0$, $A$ has $\prod\binom{n-2k_i'}{k_i-k_i'}$ distinct pairs of eigenvectors with equal eigenvalues, each spanning an eigenvector orthogonal to $\bar w$, hence in the kernel of $M$.
\end{proof}
\begin{Lemma}\label{l-sub}
The set $$\{(V_1,\dots,V_l)\in \prod \Gr_{k_i}(n)|\langle V_1,\dots,V_l\rangle\neq \mathbb{C}^n\}$$ is an irreducible projective variety.    
\end{Lemma}
\begin{proof}
If we denote a basis of $V_i$ by $v_{i,1},\dots,v_{i,k_i}$, being in the set is equivalent to $\det(v_{1,1},\dots,v_{l,k_l})=0$, which is a multilinear function of the Plücker coordinates, proving that this is a variety. The natural action of the connected algebraic group $GL_n(\mathbb{C})$ has a dense orbit, namely the $l$-tuples that span an $n-1$ dimensional subspace, and any $l-1$ of them are independent. (This is indeed an orbit: If we take a nonzero vector $v_{i,1}$ of $V_i$ for each $i$ such that $\sum v_{i,1}=0$, extend $v_{i,1}$ to a basis $v_{i,1}, \dots,v_{i, k_i}$ of $V_i$, and drop $v_{l,1}$, then we get linearly independent vectors. If we take another point $(V'_1,\dots,V'_l)$ of the orbit, corresponding vectors $v'_{i,j}$, and an invertible linear map that maps $v_{i,j}$ to $v'_{i,j}$ if $(i,j)\neq (l,1)$, then $v_{l,1}$ will automatically be mapped to $v'_{l,1}$, so $(V_1,\dots,V_l)$ will be mapped to $(V'_1,\dots,V'_l)$.) This proves irreducibility.
\end{proof}
\begin{Pro}\label{l-var}
The set $\bar X$ of $l$-tuples $(A_1,\dots,A_l)\in (\mathbb{P}M_n(\mathbb{C}))^l$, such that there exist subspaces $V_1,\dots,V_l$ with $V_i$ invariant under $A_i$, $\dim V_i=k_i$ and $\langle V_1, \dots V_l\rangle\neq \C^n$  is an irreducible projective variety.
\end{Pro}
\begin{proof}
The set of $2l$-tuples $(A_1,\dots,A_l,V_1,\dots,V_l)$ such that $V_i$ is a $k_i$ dimensional invariant subspace of $A_i$, and $\langle V_1,\dots,V_l\rangle\neq \mathbb{C}^n$ form a projective variety by the previous lemma and \cite[Lemma 3.8.]{tri}. Since it has projective spaces as fibers over the variety from the previous lemma, it is irreducible. The set in the proposition is the projection of this set to the first $l$ variables, so it inherits these properties.
\end{proof}
\begin{Pro}
The degree of the smallest degree polynomial defining $\bar X$ is $\sum\limits_{i=1}^l \binom{n}{2}\binom{n-2}{k_i-1}\prod\limits_{j\neq i}\binom{n}{k_j}$
\end{Pro}
\begin{proof}
Fix $A_2,\dots, A_l$ generically: none of them have multiple eigenvalues, and for any $V_2, \dots , V_l$ with $V_i$ being a $k_i$ dimensional invariant subspace of $A_i$, $\dim\langle V_2, \dots,V_l\rangle=\sum\limits_{i=2}^lk_i$. $(A_1,\dots, A_l)\in\bar X$ if and only if $A_1$ has a $k_1$ dimensional invariant subspace $V_1$ that nontrivially intersects the subspace generated by some invariant subspaces $V_2,\dots, V_l$. If we fix $V_2,\dots V_l$, the condition on $V_1$ is to nontrivially intersect with some $k_1$ codimensional subspace. Similarly to \ref{l-var}, one can prove that the set of matrices which have a $k$ dimensional invariant subspace that nontrivially intersects a given $k$ codimensional subspace $V$ is defined by an irreducible polynomial; let's call it $Q_V$. (Proof: The set $$\{W\in \Gr_{k}(n)|\langle W,V\rangle\neq \C^n\}$$ is a variety because it is the intersection of the set from \ref{l-sub} (for $k_1=k, k_2=n-k$) and $V_2=V$, and it is irreducible because the connected algebraic group $$\{C\in GL_n(\C)|CV=V\}$$ acts transitively on it. The set $$\{(A,W)\in M_n(\C)\times \Gr_k(n)|AW\subseteq W, \langle W,V\rangle\neq\C^n\}$$ is a variety as it is the intersection of the previous set and the set from \cite[Lemma 3.8.]{tri}, and it has projective spaces as fibers over the previous set, hence it is irreducible. The set in question is the projection of this set to the first variable.)

Observe that deg $Q_V$ doesn't depend on $V$.

Now we define $T_k(-,-)$ as in \cite{tri}, which is irreducible, a homogeneous polynomial of degree $\binom{n}{2}\binom{n}{k}\binom{n-2}{k-1}$ as a polynomial in the entries of the first variable matrix, and is $0$ if and only if the first matrix has a $k$ dimensional and the second matrix has a $k$ codimensional invariant subspace that nontrivially intersect each other. If we fix a $B$ with distinct eigenvalues, $T_k(-,B)$ has degree $\binom{n}{2}\binom{n}{k}\binom{n-2}{k-1}$, and it is clearly the product of the polynomials $Q_V$, where $V$ runs through all $k$ codimensional invariant subspaces of $B$, and by \ref{squarefree}, we may assume that all factors appear only once. Hence $\deg Q_V=\binom{n}{2}\binom{n-2}{k-1}$. The $A_1$-degree of the polynomial we're looking for is the degree of $Q_V$ times the number of choices we have for the $V_i$: $\binom{n}{2}\binom{n-2}{k_1-1}\prod\limits_{i=2}^l\binom{n}{k_i}$
\end{proof}
\begin{Th}
$\bar X$ is the vanishing set of the irreducible polynomial $\hat P=P/\prod\limits_{\sum k'>1} D_{k'}^{\prod\binom{n-2k_i'}{k_i-k_i'}}$
\end{Th}
\begin{proof}
We already proved that the minimal degree polynomial whose vanishing set is $\bar X$ is an irreducible divisor of $P$ with degree\\ $\sum\limits_{i=1}^l 
\binom{n}{2}\binom{n-2}{k_i-1}\prod\limits_{j\neq i}\binom{n}{k_j}$. It clearly cannot be $D_{k'}$ for any $k'$, so it is a divisor of $\hat P$ too. But the equality
\begin{align*}\prod\limits_{\sum k'>1} D_{k'}^{\prod\binom{n-2k_i'}{k_i-k_i'}}=\\\sqrt{\left(\prod\limits_{F\neq G}\sum\limits_{i\in [l]}\left(\sum\limits_{j \in F(i)}\lambda_{i,j}-\sum\limits_{j\in G(i)}\lambda_{i,j}\right)\right)/\prod\limits_{\sum k'_i=1}D_{k'}^{\prod\binom{n-2k'_i}{k_i-k'_i}}}\end{align*} shows that the degree of $\hat P$ is (with the notation $N=\prod\binom{n}{k_i}$)\begin{align*}\binom{N}{2}-\frac{1}{2}\left(N(N-1)-\sum\limits_{i=1}^ln(n-1)\binom{n-2}{k_i-1}\prod\limits_{j\neq i}\binom{n}{k_j}\right)=\\ \sum\limits_{i=1}^l \binom{n}{2}\binom{n-2}{k_i-1}\prod\limits_{j\neq i}\binom{n}{k_j}\end{align*} which proves the statement.
\end{proof}

\end{document}